\documentclass[12pt]{amsart}

\usepackage{psfrag}
\usepackage{color}
\usepackage{tikz}
\usetikzlibrary{matrix,arrows}
\usepackage{graphicx,graphics}
\usepackage{fullpage,amssymb,amsfonts,amsmath,amstext,amsthm,amscd,verbatim,enumerate}
\usepackage[T1]{fontenc}

\begin{document}

\newtheorem{theorem}{Theorem}[section]
\newtheorem{result}[theorem]{Result}
\newtheorem{fact}[theorem]{Fact}
\newtheorem{conjecture}[theorem]{Conjecture}
\newtheorem{lemma}[theorem]{Lemma}
\newtheorem{proposition}[theorem]{Proposition}
\newtheorem{corollary}[theorem]{Corollary}
\newtheorem{facts}[theorem]{Facts}
\newtheorem{props}[theorem]{Properties}
\newtheorem*{thmA}{Theorem A}
\newtheorem{ex}[theorem]{Example}
\theoremstyle{definition}
\newtheorem{definition}[theorem]{Definition}
\newtheorem{remark}[theorem]{Remark}
\newtheorem{example}[theorem]{Example}
\newtheorem*{defna}{Definition}

\newcommand{\notes} {\noindent \textbf{Notes.  }}
\newcommand{\note} {\noindent \textbf{Note.  }}
\newcommand{\defn} {\noindent \textbf{Definition.  }}
\newcommand{\defns} {\noindent \textbf{Definitions.  }}
\newcommand{\x}{{\bf x}}
\newcommand{\z}{{\bf z}}
\newcommand{\B}{{\bf b}}
\newcommand{\V}{{\bf v}}
\newcommand{\T}{\mathbb{T}}
\newcommand{\Z}{\mathbb{Z}}
\newcommand{\Hp}{\mathbb{H}}
\newcommand{\D}{\mathbb{D}}
\newcommand{\R}{\mathbb{R}}
\newcommand{\N}{\mathbb{N}}
\renewcommand{\B}{\mathbb{B}}
\newcommand{\C}{\mathbb{C}}
\newcommand{\ft}{\widetilde{f}}
\newcommand{\dt}{{\mathrm{det }\;}}
 \newcommand{\adj}{{\mathrm{adj}\;}}
 \newcommand{\0}{{\bf O}}
 \newcommand{\av}{\arrowvert}
 \newcommand{\zbar}{\overline{z}}
 \newcommand{\xbar}{\overline{X}}
 \newcommand{\htt}{\widetilde{h}}
\newcommand{\ty}{\mathcal{T}}
\renewcommand\Re{\operatorname{Re}}
\renewcommand\Im{\operatorname{Im}}
\newcommand{\tr}{\operatorname{Tr}}
\renewcommand{\skew}{\operatorname{skew}}

\newcommand{\ds}{\displaystyle}
\numberwithin{equation}{section}

\renewcommand{\theenumi}{(\roman{enumi})}
\renewcommand{\labelenumi}{\theenumi}

\date{\today}
\title{Quasiconformality and hyperbolic skew}

\author{Colleen Ackermann}
\address{Montgomery College }
\email{colleen.ackermann@montgomerycollege.edu}

\author{Alastair Fletcher}
\address{Department of Mathematical Sciences, Northern Illinois University, DeKalb, IL 60115-2888. USA}
\email{fletcher@math.niu.edu}
\thanks{The second author was supported by a grant from the Simons Foundation, \#352034}


\maketitle

\begin{abstract}
We prove that if $f:\B^n \to \B^n$, for $n\geq 2$, is a homeomorphism with bounded skew over all equilateral hyperbolic triangles, then $f$ is in fact quasiconformal. Conversely, we show that if $f:\B^n \to \B^n$ is quasiconformal then $f$ is $\eta$-quasisymmetric in the hyperbolic metric, where $\eta$ depends only on $n$ and $K$. We obtain the same result for hyperbolic $n$-manifolds. Analogous results in $\R^n$, and metric spaces that behave like $\R^n$, are known, but as far as we are aware, these are the first such results in the hyperbolic setting, which is the natural metric to use on $\B^n$.
\end{abstract}

\section{Introduction}

\subsection{Quasiconformal and quasisymmetric maps}

There are various equivalent definitions of quasiconformal mappings in the plane: the analytic definition via Sobolev spaces, the geometric definition involving extremal length of curve families and moduli of quadrilaterals, and the metric definition using linear dilatation. We refer to, for example, \cite{FM, Hubbard} for a fuller discussion on the various characterizations of planar quasiconformal mappings.

A more recent way to define quasiconformal mappings locally was given by Hubbard \cite{Hubbard} using a skew condition on triangles. Given a topological triangle $T$ in $\C$ with vertices $v_1,v_2,v_3$, its skew is defined to be 
\begin{equation}
\label{eq:1} 
\skew (T) = \frac{ \max_{i\neq j} |v_i - v_j| }{\min _{i\neq j} |v_i - v_j|}.
\end{equation}
 Hubbard showed that if in a neighborhood $U$ of a point $z_0$ there is a constant $\sigma$ so that the image of every triangle in $U$ with skew at most $\sqrt{7/3}$ has skew at most $\sigma$, then the map is quasiconformal in $U$. The question of whether the constant $\sqrt{7/3}$ can be reduced to $1$ was also asked in \cite{Hubbard}. After partial progress in \cite{AH}, this question was positively answered in \cite{AHH}, and so quasiconformal mappings may be characterized locally as mappings that distort the skew of equilateral triangles by a bounded amount.

The skew condition is closely related to the three point condition called quasisymmetry. A map $f$ is called quasisymmetric if there is a bijective increasing homeomorphism $\eta : (0,\infty) \to (0,\infty)$ so that for every distinct triple of points $u,v,w$, we have
\[ \left | \frac{f(u)-f(v)}{f(u)-f(w)} \right | \leq \eta \left ( \left | \frac{u-v}{u-w} \right | \right ).\]
In particular, a global quasiconformal map $f:\C \to \C$ is known to be quasisymmetric, which in turn implies the skew condition above with $\sigma = \eta(1)$. On the other hand, even a conformal map that is not global may not be quasisymmetric. In \cite[p.135]{Hubbard}, it was shown that a conformal map from the unit disk to a slit disk is not quasisymmetric and fails the skew condition. Moreover, the family of conformal self-maps of the unit disk is not uniformly quasisymmetric, that is, there is no function $\eta$ that works simultaneously for all functions in the family. To see this, one can verify that if
\[ A_r(z) = \frac{z+r}{1+rz}, \quad r\in(0,1),\]
and $u=0, v=-r, w=r$ then 
\[ \left | \frac{u-v}{u-w} \right | = 1\]
but 
\[  \left | \frac{A_r(u)-A_r(v)}{A_r(u)-A_r(w)} \right | = \frac{1+r^2}{1-r^2},\]
which diverges as $r\to 1$.

In this paper we show that there is a global characterization of quasiconformal mappings in hyperbolic space in dimension at least two in terms of a skew condition on equilateral triangles in the hyperbolic metric. This has the immediate advantage of making the family of conformal self-maps of the unit disk uniformly quasisymmetric with $\eta(t)=t$. The slit disk example mentioned above will then no longer be an issue since we will be using the hyperbolic metric on the slit disk, instead of the Euclidean metric.

\subsection{Statement of results}

We will start by stating our results in the unit ball $\B^n$ in $\R^n$, for $n\geq 2$, equipped with the hyperbolic metric $\rho$.  Given a topological triangle $T \subset \B^n$ with vertices $v_1,v_2,v_3$, its hyperbolic skew is
\[ \skew_{\rho}(T) = \frac {L(T)} {\ell (T)} , \]
where
\[ L(T) = \max_{i\neq j} \rho(v_i, v_j), \quad \ell ( T) = \min _{i\neq j} \rho(v_i, v_j) .\]
An equilateral hyperbolic triangle $T$ has $\skew_{\rho}(T) = 1$. 

\begin{definition}
\label{def:1}
Let $n\geq 2$ and $\sigma \geq 1$. Then the family $\mathcal{F}_{\sigma}$ consists of homeomorphisms $f:\B^n \to \B^n$ so that $\skew_{\rho}(f(T)) \leq \sigma$ for every equilateral hyperbolic triangle $T \subset \B^n$.
\end{definition}

\begin{theorem}
\label{thm:1}
Let $n\geq 2$ and suppose that $f\in \mathcal{F}_{\sigma}$. Then $f$ is quasiconformal.
\end{theorem}

It turns out that while equilateral hyperbolic triangles of small side length are close to equilateral Euclidean triangles, it is not straightforward to immediately apply the results of \cite{AHH} to this case. The point here is that equilateral Euclidean triangles are not equilateral hyperbolic triangles and so our hypothesis that $f\in \mathcal{F}_{\sigma}$ says nothing a priori about the boundedness of the skew of the images of equilateral Euclidean triangles. The methods employed in the proof of Theorem \ref{thm:1} are analogous to those in \cite{AHH}, but modifications to the hyperbolic setting are necessary and, in fact, we are able to substantially weaken some of the geometric requirements.

For the converse, we will prove the following.

\begin{theorem}
\label{thm:2}
Let $n\geq 2$ and suppose that $f:\B^n \to \B^n$ is $K$-quasiconformal. Then $f$ is $\eta$-quasisymmetric in the hyperbolic metric with $\eta$ depending only on $n$ and $K$.
\end{theorem}

We will see that we can in fact take $\eta(t) = C\max \{ t^K, t^{1/K} \}$, which means that $f$ is power-quasisymmetric in the hyperbolic metric. This term was introduced by Trotsenko and V\"ais\"al\"a \cite{TV}.

This result is likely known by experts in the field, but we were unable to find a reference and so we include a proof here. It is well known that this result is true for quasiconformal mappings in $\R^n$, $n\geq 2$, and there has been a substantial amount of research into generalizing this to other metric spaces that are, in a sense, analogous to Euclidean spaces. 
Heinonen and Koskela \cite[Corollary 4.8 and Theorem 4.9]{HK} proved that if $X$ and $Y$ are Ahlfors $Q$-regular metric spaces, $X$ is a Loewner space, $Y$ is locally linearly connected and $f:X\to Y$ is a quasiconformal map (in the metric sense) which maps bounded sets to bounded sets, then $f$ is quasisymmetric with $\eta$ depending only on the quasiconformality constant of $f$ and the data associated to the spaces $X$ and $Y$.

We refer to \cite{HK} for the various definitions in the above statement, except to point out that  a metric space $X$ is Ahlfors $Q$-regular means that there exists a constant $C\geq 1$ so that for all balls $B(x,r)\subset X$
\[ \frac{r^Q}{C} \leq \mathcal{H}_Q( B(x,r)) \leq Cr^Q,\]
where $\mathcal{H}_Q$ denotes the $Q$-Hausdorff measure in the underlying metric space. This means that in a $Q$-regular metric space, the size of balls of radius $r$ is comparable to $r^Q$. However, in hyperbolic space this is not true: the size of balls grow exponentially with the radius and consequently the arguments of \cite{HK} do not apply in the context of interest to this paper.
It would be interesting to see to what extent the results here can be generalized to quasiconformal mappings in spaces analogous to the ball equipped with the hyperbolic metric, for example domains in $\R^n$ equipped with the quasi-hyperbolic metric.

If $M^n$ is a hyperbolic $n$-manifold, for $n\geq 2$, then by definition there is a covering map $\pi_M :\B^n \to M^n$ and an associated group of covering transformations $G_M$ acting properly discontinuously on $\B^n$, so that $M^n$ can be realized as $\B^n / G_M$. See \cite[p.348]{Ratcliffe} applied to the unit ball model of hyperbolic space. Then the hyperbolic distance $\rho_M$ can be defined via the hyperbolic distance $\rho$ on $\B^n$ and the formula
\[ \rho_M(p,q) = \inf_{\pi_M(x) = p, \pi_M(y) = q} \rho(x,y).\]
If $x$ is considered fixed with $\pi_M(x) = p$, then by the discreteness of $G_M$ we also have
\[ \rho_M(p,q) = \min_{\pi_M(y) = q} \rho(x,y).\]
We then obtain the following corollaries to Theorem \ref{thm:2}.

\begin{corollary}\label{cor:1}
Let $n\geq 2$ and let $M^n,N^n$ be hyperbolic $n$-manifolds carrying hyperbolic distance functions $\rho_M,\rho_N$ respectively. Then a homeomorphism $f:M\to N$ is $K$-quasiconformal if and only if it is $\eta$-quasisymmetric with respect to $\rho_M$ and $\rho_N$, where $\eta$ depends only on $K$ and $n$.
\end{corollary}

\begin{corollary}\label{cor:2}
Let $n\geq 2$ and let $M^n$ and $N^n$ be hyperbolic $n$-manifolds. Then a homeomorphism $f:M^n \to N^n$ is quasiconformal if and only if there is a constant $\sigma \geq 1$ such that for all equilateral hyperbolic triangles $T$ in $M^n$, we have $\skew _{\rho_N}(f(T)) \leq \sigma$.
Similarly to before we define
 \[ \skew_{\rho_N} (f (T)) =\frac{max_{i\neq j}\rho_N (v_i, v_j)} {min_{i\neq j}\rho_N (v_i, v_j)}\]
where $v_1$, $v_2$ and $v_3$ are the vertices of the topological triangle $f (T) $.
\end{corollary}

\begin{corollary}\label{cor:3}
Let $n\geq 2$ and let $M^n$ and $N^n$ be hyperbolic $n$-manifolds. Then the family $\mathcal{F}_K$ of $K$-quasiconformal maps from $M^n$ onto $N^n$ is a uniformly quasisymmetric family with respect to the hyperbolic distances on $M^n$ and $N^n$.
\end{corollary}

Note that the $n=2$ case in the corollaries above applies to hyperbolic Riemann surfaces which, via the Uniformization Theorem, are almost all Riemann surfaces.

The paper is organized as follows. In section 2, we recall some facts about quasiconformal mappings and hyperbolic geometry. In section 3, we prove Theorem \ref{thm:1}. In section 4, we prove Theorem \ref{thm:2} and its corollaries.

The authors would like to thank Vlad Markovic for helpful conversations on the topic of this paper.

\section{Preliminaries}

\subsection{Hyperbolic Geometry}

Let $n\geq 2$ and let $\B^n$ be the unit ball in $\R^n$. We equip $\B^n$ with the hyperbolic density
\begin{equation}\label{eqn:1}
 \lambda (x) \: |dx| = \frac{ 2|dx|}{1-|x|^2}. 
\end{equation} 
The hyperbolic metric on $\B^n$ is defined by
\[ \rho (u,v) = \inf \int_{\gamma} \lambda(x) \: |dx|,\]
where the infimum is taken over all paths in $\B^n$ joining $u$ and $v$. The infimum is achieved for circular arcs which, if extended to $\partial \B^n$, cut through $\partial \B^n$ perpendicularly. We will denote by $B_\rho(x_0,r)$ the open hyperbolic ball of radius $r>0$ centred at $x_0 \in \B^n$. Balls in other metric spaces will use similar notation.

\subsection{In dimension two}

We refer to \cite{BM} for a reference to the theory of hyperbolic geometry in dimension two.
The formula for the hyperbolic metric on the unit disk $\D $ is given by
\[ \rho(z,w) = \log \frac{ 1+ \left | \frac{z-w}{1-\overline{w}z} \right | }{1- \left | \frac{z-w}{1-\overline{w}z} \right | }, \quad z,w\in \D.\]
Isometries of the hyperbolic metric are given precisely by M\"obius transformations which preserve the unit disk.

The hyperbolic metric can be defined on any simply connected proper sub-domain $U$ of $\C$ via a Riemann map $\varphi :U \to \D$. We then define the hyperbolic density on $U$ by
\[ \lambda_U(x) = \lambda_{\D}( \varphi(z)) |\varphi '(z)| \]
where $\lambda_{\D} $ is defined in formula (\ref{eqn:1}), and the hyperbolic metric on $U$ by integrating $\lambda_U$. The hyperbolic metric can be defined on any plane domain and, more generally, any Riemann surface that is not covered by the sphere or plane via the Uniformization Theorem.

An equilateral hyperbolic triangle $T$ has three vertices $v_1,v_2$ and $v_3$ and three edges made by geodesic segments of equal length joining the vertices. The side length $r$ of $T$ determines the interior angles. Applying a M\"obius map to send one of the vertices to $0$ and another to $x>0$, the remaining vertex must be sent to $xe^{i\alpha}$ for some $\alpha$. Since 
\[ r = \rho(0,x) = \rho(x,xe^{i\alpha}),\]
we can compute that
\begin{equation}
\label{eq:alphax} 
\alpha = \cos^{-1} \left ( \frac{1+x^2}{2} \right).
\end{equation}
We can express $\alpha$ in terms of $r$ by using the relationships
\[ r = \log \frac{1+x}{1-x}, \quad x = \frac{e^r -1}{e^r+1}\]
to see that
\begin{equation}
\label{eq:alphar}
\alpha = \cos^{-1} \left ( \frac{1+\tanh^2(r/2)}{2} \right).
\end{equation}
As $r\to 0$, we observe that $\alpha \to \pi/3$ and so small equilateral hyperbolic triangles are close to equilateral Euclidean triangles. Evaluating when $r=1$, we have the following lemma.

\begin{lemma}
\label{lem:tri0}
If an equilateral hyperbolic triangle has side length $0<r\leq 1$, then the internal angles satisfy $2\pi/7 < \theta < \pi/3$.
\end{lemma}

The centroid of an equilateral hyperbolic triangle $T$ can be found by applying a M\"obius map $A$ to send the vertices to $t,t\omega$ and $t\omega^2$, where $t>0$ and $\omega = e^{2\pi i /3}$. Then $0$ is the centroid of the resulting triangle, and is the common intersection point of the geodesic segments joining a vertex to the midpoint of the opposite side. Applying $A^{-1}$, we see that $A^{-1}(0)$ is the centroid of $T$.

We call a collection of equilateral hyperbolic triangles $T_1,\ldots,T_m$ of the same side length $r$ in $\D$ a {\it chain} if $T_j$ and $T_{j+1}$ have a common side for $j=1,\ldots, m-1$. We allow the triangles in the chain to overlap.

\subsection{Quasihyperbolic metric}

A metric that is related to the hyperbolic metric, but can be defined on any proper subdomain $U$ of $\R^n$, is the quasihyperbolic metric given by density
\[ \delta_U(x) |dx| = \frac{|dx|}{d(x,\partial U)},\]
where $d(x,\partial U)$ denotes the Euclidean distance from $x$ to the boundary of $U$. 
The quasihyperbolic metric is denoted $q_U$ and obtained by integrating the density $\delta_U$.

The hyperbolic and quasi-hyperbolic metrics are bi-Lipschitz equivalent on simply connected proper subdomains $U$ of $\C$. In fact, it follows from the Koebe $1/4$-Theorem that
\[ \frac{\delta_U(z)}{2} \leq \lambda_U(z) \leq 2\delta_U(z),\]
for all $z\in U$. This is not true in general, considering for example the punctured disk. In dimension three and greater, we can only define the hyperbolic metric on balls and half-spaces. This is a consequence of the generalized Liouville's Theorem, see for example \cite[Theorem I.2.5]{Rickman}, which says that the only $1$-quasiregular mappings in $\mathbb{R} ^ n $ with $n\geq 3 $ are (restrictions of) M\"obius transformations. Consequently, the quasihyperbolic metric plays the role of the hyperbolic metric in function theory in higher dimensions.

\subsection{Quasiconformal mappings}

As remarked at the outset of this paper, there are various equivalent definitions of quasiconformal mappings in $\R^n$, $n\geq 2$. We will give the analytic definition and the metric definition.

\begin{definition}[Analytic Definition]
A {\it quasiconformal} mapping in a domain $U\subset \R^n$ for $n\geq 2$ is a homeomorphism in the Sobolev space $W^1_{n,loc}(U)$ where there is a uniform bound on the distortion, that is, there exists $K\geq 1$ such that
\[|f'(x)|^n \leq KJ_f(x)\]
almost everywhere in $U$. 
The minimum such $K$ for which this inequality holds is called the {\it outer  dilatation} and denoted by $K_O(f)$. As a consequence of this, there is also $K' \geq 1$ such that 
\[J_f(x) \leq K' \inf_{|h|=1}|f'(x)h|^n\]
holds almost everywhere in $U$. The minimum such $K'$ for which this inequality holds is called the {\it inner  dilatation} and denoted by $K_I(f)$. If $K (f)= \max \{K_O(f), K_I(f) \}$, then $K(f)$ is the {\it  maximal dilatation} of $f$. A $K$-quasiconformal mapping is a quasiconformal mapping for which $K(f) \leq K$.
\end{definition}

\begin{definition}[Metric Definition]
\label{def:metric}
Let $n\geq 2$ and let $U \subset \R^n$ be a domain. Then $f:U \to \R^n$ is quasiconformal  if and only if there exists a constant $H$ such that
\[ \limsup_{r\to 0 } H(x,r) \leq H\]
for all $x\in U$, where 
\[ H(x,r) = \frac{ \max_{|x-y| =r} |f(x)-f(y)|}{ \min_{|x-y| = r} |f(x)-f(y)|}.\]
\end{definition}

The Analytic and Metric Definitions are equivalent and, in fact, $H \leq K^{2/n}$ and $K\leq H^{n-1}$, see \cite[p.109]{IM}.

In this paper, we will be interested in a hyperbolic version of linear distortion. We therefore define for $x\in \B^n$ and $r>0$
\[ H_{\rho} (x,r) = \frac{ L_{\rho}(x,r) }{ \ell_{\rho}(x,r) } ,\]
where
\[ L_{\rho}(x,r) = \max_{\rho(x,y) =r } \rho(f(x),f(y)) , \quad \ell_{\rho}(x,r) = \min_{\rho(x,y) =r } \rho(f(x),f(y)).\]
The following result on the distortion of the hyperbolic metric was proved by Gehring and Osgood \cite{GO} for the quasihyperbolic metric with a constant depending on $n$ and $K$, then improved to a dimension independent version by Vuorinen \cite{Vu} (see also \cite[Corollary 12.20]{Vuorinen}). For our purposes with the hyperbolic metric, we just note that the hyperbolic and quasihyperbolic metrics are bi-Lipschitz equivalent on the unit ball.

\begin{theorem}
\label{thm:qc1}
Let $n\geq 2 $ and let $\B^n$ be the unit ball in $\R^n$ equipped with the hyperbolic metric $\rho$. Then if $f:\B^n \to \B^n$ is a $K$-quasiconformal mapping,
\[ \rho(f(x),f(y)) \leq C_1 \max \{ \rho(x,y)^{1/K}, \rho(x,y) \},\]
and
\[ \rho(f(x),f(y)) \geq C_2 \min \{ \rho(x,y)^{K}, \rho(x,y) \},\]
where $C_1,C_2$ are constants that depend only on $K$.
\end{theorem}

We will also need the following result which characterizes quasiconformal mappings as local quasisymmetric mappings in a quantitative way. This result is due to V\"ais\"al\"a \cite[Theorem 2.4]{Va} and is slightly reformulated for our purposes (see also \cite[Theorem 11.14]{Heinonen}).

\begin{theorem}
\label{thm:qc2}
Let $n\geq 2$, suppose that $U \subset \R^n$ is open and suppose that $f:U \to \R^n$ is $K$-quasiconformal. Suppose also that $x_0 \in U$, $0<\lambda <1$ and $r>0$ so that $B(x_0,r) \subset U$. Then $f$ restricted to $B(x_0,\lambda r)$ is $\xi$-quasisymmetric, where $\xi$ depends only on $n,K$ and $\lambda$.
\end{theorem}

This result is informally called the egg yolk principle: the smaller ball is the yolk, the larger ball is the egg and, however wildly $f$ behaves near the edge of the egg, it is relatively well-behaved on the yolk.

\section{Hyperbolic Equilateral Triangles}

In this section, we will prove Theorem \ref{thm:1}. In dimension three and higher, the proof is easier and so we will deal with this case first and then move to the dimension two case. Throughout, if $n\geq 2$, denote by $\B^n$ the unit ball in $\R^n$, by $\rho$ the hyperbolic metric on $\B^n$ and by $\mathcal{F}_{\sigma}$, the family of homeomorphisms $f:\B^n \to \B^n$ satisfying the skew condition with constant $\sigma \geq 1$, recalling Definition \ref{def:1}. 

\begin{proof}[Proof of Theorem \ref{thm:1} with $n\geq 3$]
Fix $n\geq 3$. Choose $r_0$ small enough so that an equilateral hyperbolic triangle in $\B^n$ with side length $r\leq r_0$ has interior angles at least $\pi/3 - \delta$ for some small fixed $\delta$. For such an $r$, denote by $S_r$ the boundary of $B_{\rho}(0,r)$ in $\B^n$. 

Suppose $f \in \mathcal{F}_{\sigma}$ and for now assume that $f$ fixes $0$. Since $S_r$ is compact, $L_{\rho}(0,r)$ and $\ell_{\rho}(0,r)$ are achieved on $S_r$ at, say, $x_1$ and $x_0$ respectively.

Consider all equilateral hyperbolic triangles which have two vertices at $0$ and $x_1$. The locus of all possible locations for the third vertex is an $(n-2)$-sphere $\Sigma_1$ contained in $S_r$. Similarly, $\Sigma_0$ is the locus of all possible locations for the third vertex of an equilateral hyperbolic triangle with vertices at $0$ and $x_0$.

If $\Sigma_0$ and $\Sigma_1$ intersect, then we can choose $x_2$ to be an intersection point. Otherwise we choose $x_2 \in \Sigma_1$ to be a closest point to $x_0$. We then define $\Sigma_2$ analogously for $x_2$ and check whether $\Sigma_2$ intersects $\Sigma_0$. Continuing in this fashion, we build a chain of at most four triangles where the initial triangle has vertices $0,x_1,x_2$ and the final triangle has vertices including $0$ and $x_0$. The reason we can do this  with at most four triangles is that the interior angles of each triangle are at least $\pi/3 - \delta$.

Finally, since $f\in \mathcal{F}_{\sigma}$, we obtain
\[ L_{\rho}(0,r) \leq \sigma^4 \ell_{\rho}(0,r),\]
for all $r\leq r_0$. Hence $f$ is quasiconformal at $0$. For any other point $x\in \B^n$, we can apply M\"obius maps $A_1,A_2$ which send $x$ and $f(x)$ to $0$ respectively and then apply the above argument to $A_2\circ f \circ A_1^{-1}$.
\end{proof}

We next turn to the dimension two case. We first need some preliminary results on hyperbolic geometry.

\begin{lemma}
\label{lem:tri1}
There exists $\delta>0$ so that any equilateral hyperbolic triangle $T$ in $\D$ of side length $r\leq 1$ has the property that $B_{\rho}(c,   2\delta r) \subset T$, where $c$ is the centroid of $T$.
\end{lemma}

\begin{proof}
Given any equilateral hyperbolic triangle $T$ of side length $r$, we may apply a M\"obius transformation $A$ so that the vertices of $A(T)$ lie at the points $t, \omega t, \omega^2 t$, where $t>0$ and $\omega = e^{2\pi i/3}$. By a direct computation, the quantities $r$ and $t$ are related via
\[ r = \rho(t,t\omega) = \log \frac{ 1+ \frac{t\sqrt{3}}{\sqrt{1+t^2+t^4}} }{ 1- \frac{t\sqrt{3}}{\sqrt{1+t^2+t^4}} }.\]
We see that $r = 2\sqrt{3}t+o(t)$ as $t\to 0$.

By the formula for the midpoint of a hyperbolic geodesic segment, see \cite[Proposition 3.2]{CFY}, the hyperbolic midpoint of $t\omega$ and $t\omega^2$ occurs on the negative real axis at
\[ \frac{ \sqrt{1+t^2+t^4} -1 - t^2}{t}.\]
This implies that any Euclidean ball of radius less than $R(t) = \frac{1+t^2 - \sqrt{1+t^2+t^4}}{t}$ centred at $0$ is contained in $T$. Therefore any hyperbolic ball of radius less then $\widetilde{R}(t):= \log \frac{1+R(t)}{1-R(t)}$ centred at $0$ is contained in $T$.

Now, $R$ is an increasing function of $t$ with $R(t) = t/2 + o(t)$ as $t\to 0$ (as one would expect since small equilateral hyperbolic triangles are close to small Euclidean triangles) and $\lim_{t\to 1} R(t) = 2-\sqrt{3}$. Hence $\widetilde{R}$ is also increasing with $\widetilde{R}(t) = t + o(t)$ as $t\to 0$.
Consequently, if the side length $r$ of $T$ is at most $1$, then we can find $\delta >0$ so that $B_{\rho}(0,  2\delta r)$ is contained in $T$.
\end{proof}

Given an equilateral triangle of side length $r\leq 1$, we will denote by $B_{\delta}(T)$ the ball $B_{\rho}(c,\delta r)$, where $\delta$ is from Lemma \ref{lem:tri1}. Then if $p\in B_{\delta}(T)$, we have $B_{\rho}(p,\delta r) \subset B_{\rho}(c,2\delta r) \subset T$.

If $E\subset \D$ is closed and $z\in \D \setminus E$, the hyperbolic distance between $z$ and $E$ is
\[ \rho(z,E) = \min \{ \rho(z,w) : w\in E \}.\]

\begin{lemma}
\label{lem:tri2}
Let $T$ be an equilateral hyperbolic triangle in $\D$ with side length $r\leq 1$ and let $p \in \D$. Then there exists a chain of equilateral hyperbolic triangles $T_1,\ldots, T_m$ with side length $r$, 
$T_1 = T$, $p \in T_m$ and moreover $m\leq M$, where $M = \max \{ 7, 700\rho( p,T)/r \}$.
\end{lemma}

\begin{proof}
Without loss of generality, we may apply a M\"obius map so that $T$ has vertices $t,t\omega, t\omega^2$, where $t>0$ and $\omega = e^{2\pi i/3}$ and $T $ has centroid $0$. Further, we may assume that $-\pi/3 \leq \arg p \leq \pi/3$, otherwise apply a rotation permuting the vertices of $T$.

Since $r\leq 1$, Lemma \ref{lem:tri0} implies that the internal angles $\alpha$ of $T$ are at least $2\pi/7$. Consequently, if we form a chain of triangles by rotating $T$ in the clockwise direction through angle $\alpha$ about $t$, then by the time we add in the seventh triangle, we will intersect $T$.

Let $U$ be an open $r/100$ neighborhood of $T$ in the hyperbolic metric and let
\[ U' =  \{ z: z\in U \text{ and } \arg z \in [-\pi/3, \pi/3 ] \}.\]
Then the collection $\mathcal{C}$ of seven triangles obtained by forming the chain around the point $t \in T$ covers $U'$. If $p$ lies in $T$ or this chain, then we are done. Otherwise, consider a geodesic segment realizing the distance $\rho(p,T)$. This segment must cross $U'$ and consequently there is a triangle $T_1 \in \mathcal{C}$ satisfying
\[ \rho(p,T_1) \leq  \rho(p,T) - \frac{r}{100}.\]
Repeating this process, we are able to construct a chain of triangles as required. Each step requires at most seven triangles and so the maximum number required is
$\frac{700\rho(p,T)}{r}$.
\end{proof}

\begin{lemma}
\label{lem:angle}
Let $0<t\leq 1$ and let $T$ be the hyperbolic triangle with vertices $v_1,v_2,v_3$ so that $v_1=0$, $\arg (v_2) = e^{i\pi/3}$, $\arg (v_3) = e^{-i\pi/3}$ and $\rho(v_1,v_2) = \rho(v_1,v_3) = t$. Then given $\epsilon >0$, there exists $\xi>0$ so that if $\rho(v_i,w_i) < t\xi$ for $i=1,2,3$ and if $\phi$ denotes the angle $\angle w_2w_1w_3$ of the hyperbolic triangle $T'$ with vertices $w_1,w_2,w_3$, then $|2\pi/3 - \phi|<\epsilon$.
\end{lemma}

\begin{proof}
By the hyperbolic Law of Cosines, 
\begin{equation}
\label{eq:angle}
\cos \phi = \frac{ \cosh ( \rho(w_1,w_2) ) \cosh ( \rho(w_1,w_3) ) - \cosh ( \rho(w_2,w_3) ) }{ \sinh (\rho(w_1,w_2) ) \sinh (\rho (w_1,w_3) ) }.
\end{equation}
 Clearly by construction the angle $\angle v_2v_1v_3$ is $2\pi/3$, and so
replacing the $w_i$ by the $v_i$ in this formula, we obtain $\cos(2\pi/3) = -1/2$. 
By the hypotheses and the triangle inequality,
\[ (1-2\xi)t <\rho ( w_1,w_i) <(1+2\xi) t\]
for $i=2,3$.
Writing $h(t) = \rho(v_2,v_3)$, by the triangle inequality,
\[ h(t) - 2\xi t < \rho(w_2,w_3) < h(t) + 2\xi t.\]
We therefore see
\[ \frac{ \cosh^2( (1-2\xi)t ) - \cosh( h(t) + 2\xi t) }{\sinh^2((1+2\xi)t) }< \cos \phi <\frac{ \cosh^2( (1+2\xi)t ) - \cosh( h(t) - 2\xi t) }{\sinh^2((1-2\xi)t) }.\]
By the continuity of the functions involved here and since the limit as $\xi \to 0$ of both left and right hand sides is $-1/2 = \cos(2\pi/3)$, the claim follows.
\end{proof}

Our final preliminary result we need is to construct a certain self-map of the unit disk which fixes the origin, acts as a rotation on each circle centered at the origin and generates equilateral hyperbolic triangles. We recall that a homeomorphism $f:\D\to \D$ is called {\it locally quasiconformal} if and only if for each compact siubset $E \subset \D$, $f|_E$ is quasiconformal. See for example \cite{Li}. In particular, $|\mu_f(z)|$ is allowed to converge to $1$ as $|z|\to 1$.

\begin{lemma}
\label{lem:map}
The map $R_0 : \D \to \D$ defined in polar coordinates by 
\[ R_0(te^{i\theta}) = t \exp \left [ i \left ( \theta + \cos^{-1} \left ( \frac{1+t^2}{2} \right ) \right ) \right ] \]
is locally quasiconformal, fixes $0$ and, moreover, for any $w \in \D \setminus \{ 0 \}$, the hyperbolic triangle with vertices $0,w$ and $R_0(w)$ is equilateral.
\end{lemma}

\begin{proof}
It is clear that $R_0$ acts as a rotation on each circle centered at the origin. Moreover, the claim on equilateral hyperbolic triangles follows from \eqref{eq:alphax}. The formula for the complex dilatation of $R_0$ in terms of polar coordinates is
\[ \mu_{R_0} = e^{2i\theta} \left ( \frac{ (R_0)_t + \frac{i}{t} (R_0)_{\theta}}{ (R_0)_t - \frac{i}{t} (R_0)_{\theta} } \right ).\]
Via elementary computations we have
\[ (R_0)_{\theta} = it  \exp \left [ i \left ( \theta + \cos^{-1} \left ( \frac{1+t^2}{2} \right ) \right ) \right ]\]
and 
\[ (R_0)_t =  \left ( 1 -\frac {2it^2}{\sqrt{(3+t^2)(1-t^2) } } \right)  \exp \left [ i \left ( \theta + \cos^{-1} \left ( \frac{1+t^2}{2} \right ) \right ) \right ].\]
Hence we obtain
\[ | \mu_{R_0} (te^{i\theta}) | = \frac{t^2}{ |\sqrt{ (3+t^2)(1-t^2)} - it^2 |} = \frac{t^2}{\sqrt{3-2t^2}}.\]
Hence $R_0$ is quasiconformal on each compact subset of $\D$, but $|\mu_{R_0} (z)| \to 1$ as $|z|\to 1$.
\end{proof}

\begin{definition}
\label{def:rz}
For $w\in \D$, let $A_w(z) = \frac{z-w}{1-\overline{w}z}$. Then define $R_w :\D \to \D$ by $R_w = A_w^{-1} \circ R_0 \circ A_w$.
\end{definition}

The key property is that for any $z\neq w$, the hyperbolic triangle with vertices $w,z,R_w(z)$ is equilateral.

With these results in hand, we can prove the remaining case of Theorem \ref{thm:1}.

\begin{proof}[Proof of Theorem \ref{thm:1} when $n=2$]
Let $\sigma \geq 1$ and suppose that $f \in \mathcal{F}_{\sigma}$. Let $T$ be an equilateral hyperbolic triangle of side length $r\leq 1$. Then by Lemma \ref{lem:tri1}, we know that $B_{\delta}(T) \subset T$ for a constant $\delta>0$ independent of $r$.  The definition of $B_{\delta} (T) $ is given directly after the proof of Lemma \ref{lem:tri1}.

We will prove the theorem in a number of steps. Given $p$ close to the centroid of $T$, we find a chain of small equilateral triangles connecting a side of $T$ to $p$. Then we find a particular small equilateral triangle close to $p$, and show that this construction implies that the image $f(T)$ contains a disk of a definite size, relative to the side length of $f(T)$, centred at $f(p)$. Finally, we show how this implies that $f$ satisfies the metric definition of quasiconformality.

{\bf Step 1: constructing a chain of small triangles.}
 Let $p\in B_{\delta}(T)$ and let $ n\in\mathbb{N} $. We will specify how large $ n$ must be later.
Select the side of $T$ which realizes $L(f(T))$, the maximum distance between two vertices of the topological triangle $f (T) $, and subdivide this side into $r/n$ segments of equal length.  Let $v, w $ be the endpoints of the segment whose image has the largest length and $T_1$ be the equilateral triangle in $T$ which has one side with vertices $v,w$. Therefore
\[ L(f(T)) \leq n \rho ( f(v) , f(w) ).\]
Apply Lemma \ref{lem:tri2} to find a chain of triangles $T_1,\ldots, T_m$ of side length $r/n$ with $p\in T_m$. 
Since $T$ has side length length less than or equal to 1 which implies $\rho(p,T_1)\leq 1$, we can achieve this with $m \leq M = 700n$ triangles.
Since $f\in \mathcal{F}_{\sigma}$, we find by induction that if $v', w'$ is any other side in the chain,
\[ \rho( f(v), f(w) ) \leq \sigma ^M \rho (f(v'), f(w') )\]
and hence
\[ L(f(T)) \leq n \sigma^M \rho (f(v') , f(w') ).\]
Choose $v',w'$ to be any two vertices of $T_m$ that are different from $p$ (typically $p$ will not be a vertex of $T_m$). Then for
one of $v',w'$, denoted by $q$, we are guaranteed by the triangle inequality to have
\[ L(f(T)) \leq 2n \sigma^M \rho ( f(p) , f(q) ),\]
and
\[ \rho (p,q) \leq r/n.\]

{\bf Step 2: constructing a small equilateral triangle.}
Denote by $\mu$ the distance from $f(p)$ to $\partial f(T)$. 
We can realize $\mu$ as the length of a hyperbolic geodesic segment joining $f(p)$ to $\partial f(T)$. Let $\gamma$ be the pre-image of this geodesic segment and further denote by $\gamma_1$ the component of $\gamma \cap B_{\rho}(p,\delta r)$ that contains $p$.
Next, denote by $\gamma_2$ the curve $R_p(\gamma_1) \cup R_p^{-1}(\gamma_1)$. We observe that since $R_p$ is locally quasiconformal, so is $R_p^{-1}$ and hence $\gamma_2$ really is a curve. For $t\in \gamma_2$, we can find $s\in \gamma_1$ that arises as its pre-image under either $R_p$ or $R_p^{-1}$. Then
\begin{equation}
\label{eq:thm11}
\rho( f(t),f(p)) \leq \sigma \rho( f(s),f(p) ) \leq \sigma \mu,
\end{equation}
since the triangle with vertices $s,t$ and $p$ is equilateral by Lemma \ref{lem:map}.

Next, we need to ensure that $\gamma_2$ is well-behaved near its endpoints. To that end, we will slightly enlarge the curve, while maintaining an inequality similar to \eqref{eq:thm11}.
Given $\epsilon <1$, take the corresponding $\xi$ from Lemma \ref{lem:angle}.
Denote by $a,b$ the endpoints of $\gamma_2$ on $\partial B_{\rho}(p,\delta r)$ and let $B_a,B_b$ be the disks $B_{\rho}(a,\delta r \xi)$ and $B_{\rho}(b,\delta r \xi)$ respectively. Write $\gamma_{2a}$ and $\gamma_{2b}$ for the components of $\gamma_2 \cap B_a$ and $\gamma_2 \cap B_b$ that have endpoints at $a$ and $b$ respectively.

We focus on extending $\gamma_{2a}$ in $B_a$ and will perform an analogous construction for $\gamma_{2b}$ in $B_b$.
Denote by $a'$ the endpoint of $\gamma_{2a}$ on $\partial B_a$. 
Use the hyperbolic geodesic through $a$ tangent to $B_{\rho}(p,\delta r)$ at $a$ to divide $B_a$ into two parts and then
use hyperbolic rotations through angle $\pi /3$ about $a$ to divide $B_a$ into six sectors, each of which has angle $\pi/3$ seen from $a$. Let $S_a$ denote the middle sector that has no intersection with $B_{\rho}(p,\delta r)$. By Lemma \ref{lem:tri0} and Lemma \ref{lem:map}, $R_a$ acts on $\partial B_a$ by rotating through an angle between $2\pi/7$ and $\pi/3$. Hence there exists $n\in \Z$ with $|n| \leq 3$ so that $R_a^n(a')$ lies in the sector $S_a$. Let the image of $\gamma_{2a}$ under $R_a^n$ be denoted by $\gamma_{3a}$.

 Let $t \in \gamma_{3a}$  and $t_0 \in \gamma_{2a}$  such that $R_a^n(t_0) = t $. If $t_k$ denotes the image of $t_0$ under $R_a^k$, then $a, t_{k-1}$ and $t_k$ form an equilateral triangle and so
\[ \rho( f(t_{k-1}), f(t_k) ) \leq \sigma \rho ( f(t_{k-1}), f(a) ),\]
since $f\in \mathcal{F}_{\sigma}$.
By the triangle inequality and \eqref{eq:thm11} we have
\[ \rho( f(a), f(t_0) ) \leq \rho (f(a), f(p) ) + \rho ( f(p), f(t_0) ) \leq 2\sigma \mu.\]
Then we conclude that since $|n| \leq 3$,

\begin{align} 
\label{eq:t_ineq}
\rho ( f(t), f(p) ) & \leq \rho( f(a), f(p )) + \rho(f(a), f(t) )\\ 
\nonumber &\leq \rho ( f(a) , f(p) ) + \sigma^3 \rho ( f(a) , f(t_0) )\\
\nonumber &\leq \sigma \mu + 2\sigma^4 \mu \\
\nonumber &= \sigma \mu( 1+2\sigma^3).
\end{align}

If  $\delta ' \leq \xi\delta$ is chosen appropriately so that the endpoints of the geodesic segments forming $S_a$ intersect $\partial B_a$ on $\partial B_{\rho}(p,(\delta + \delta ')r)$, then $\gamma_{3a}$ must also intersect $\partial B_{\rho}(p,(\delta + \delta ')r)$.

We make the same construction using the point $b$ instead of the point $a$ and then define $\gamma_3$ to be the connected component of $(\gamma_2 \cup \gamma_{3a} \cup \gamma_{3b} ) \cap B_{\rho}(p, (\delta + \delta ')r)$  that includes $p$. By the construction above, we have
\[ \rho( f(t), f(p) ) \leq \sigma \mu ( 1+2\sigma ^3)\]
for all $t\in \gamma_3$.

We will use this curve $\gamma_3$ to find the required equilateral triangle. 
Recall the point $q$ from Step 1, and let $B_q$ be the smallest disk centred at $q$ which contains $B_{\rho}(p, \delta r)$. 
We choose $n$ large enough in Step 1 so that $1/n < \delta \xi$ and $B_q \subset B_{\rho} ( p, (\delta + \delta ' )r)$.
Let $\gamma_4$ be the connected component of $\gamma_3 \cap B_q$ with endpoints $A \in B_a \cap \partial B_q$ and $B\in B_b \cap \partial B_q$.

By our construction,
the angle made by the geodesics joining $A$ to $q$ and $B$ to $q$ make an angle in $(\pi/3,\pi)$. To see this, take a M\"obius map $M$ which moves $p$ to $0$. Then since $\rho(p,q) < \xi \delta r$, $A\in B_a, B\in B_b$ and $\epsilon <1<\pi/3$ (recall $\epsilon $ was selected in the second paragraph of Step 2) applying Lemma \ref{lem:angle} to $M(q), M(A)$ and $M(B)$ gives the claim.

Since we may assume that $(\delta + \delta ' )r <1$, by Lemma \ref{lem:tri0} $R_q$ acts on $\partial B_q$ by a rotation with angle strictly between $2\pi/7$ and $\pi/3$.
It follows that the images $R_q(A)$ and $R_q(B)$ will separate $A$ and $B$ on $\partial B_q$. Consequently $R_q(\gamma_4)$ must intersect $\gamma_4$.

We then take $t_1$ to be an intersection point, $t_2$ to be its pre-image under $R_q$ and obtain, via Lemma \ref{lem:map}, an equilateral triangle with vertices $q,t_1$ and $t_2$. It is possible that these vertices coincide if $\gamma_4$ passes through $q$ and then we obtain the trivial triangle.  Note by \eqref{eq:t_ineq} we have
\[\rho( f(t_i), f(p) ) \leq \sigma \mu ( 1+2\sigma ^3)\]
for $i\in \{1, 2\} $.

{\bf Step 3: a disk of a definite size in the image.} 
In Steps 1 and 2, given $p\in B_{\delta}(T)$, we found an equilateral triangle with vertices $q,t_1,t_2$ and constants $D_1,D_2$ so that
\begin{equation}
\label{eq:step2} 
\rho( f(t_j), f(p) ) \leq D_1\mu, \quad \rho(f(p) , f(q) ) \geq D_2 L(f(T)),
\end{equation}
recalling that $\mu$ is the distance from $f(p)$ to $\partial f(T)$. If $t_1=t_2 =q$, then by \eqref{eq:step2}
\[ D_2 L(f(T)) \leq \rho ( f(p) , f(q) ) \leq D_1 \mu,\]
which implies there is a disk in $f(T)$ centred at $f(p)$ of radius at least $D_2 L(f(T)) / D_1$.
Otherwise, we have by the triangle inequality and the assumption that $f\in \mathcal{F}_{\sigma}$ that
\begin{align*}
\rho ( f(p ) , f(q) ) - \rho( f(t_1),f( p ) ) &\leq \rho ( f(t_1) , f(q) ) \\
&\leq \sigma \rho ( f(t_1) , f(t_2) ) \\
&\leq \sigma( \rho ( f(t_1) , f(p ) ) + \rho ( f(p ) , f(t_2) ) ) .
\end{align*}
Now using \eqref{eq:step2}, we obtain
\[ D_2L(f(T)) - D_1 \mu \leq 2\sigma D_1 \mu,\]
and so
\[ \mu \geq \frac{ D_2 L(f(T)) }{(2\sigma + 1) D_1}.\]
Again we conclude that there is a disk of size $\beta$ centred at $f(p)$ in $f(T)$, where $\beta$ depends only on $\sigma$ and $L(f(T))$ (note that the side length of $T$ is $r\leq 1$ and $\beta$ does not depend on $r$ once we have fixed this upper bound).

{\bf Step 4: Showing $f$ is quasiconformal.} We first assume that $f$ fixes $0$. Let $r\leq 1$. By pre-composing $f$ with a rotation, we may assume that $L_{\rho}(0,r)$ is taken at $z_0$ on the positive real axis, where $\rho(0,z_0) = r$. Let $T_1$ be the hyperbolic equilateral triangle with vertices $0,z_0$ and $z_0 e^{i\alpha}$ and centroid $c_0$. By Step  3, $f(T_1)$ contains a disk centred at $c_0$ with radius at least $\beta L(f(T_1))$.

There exists a hyperbolic isometry which maps $z_0$ and $c_0$ to $c_0$ and $0$ respectively and $T_1$ onto an equilateral hyperbolic triangle $T_2$. Moreover, $0 \in B_{\delta}(T_2)$ because 0 is the centroid of $T_1$ . Since one vertex of $T_2$ is contained in $T_1$ and the other two are outside, it follows that $L(f(T_2)) \geq \beta L(f(T_1))$. Since $0\in B_{\delta}(T_2)$, we can apply Step 3 again to see that $f(T_2)$ contains the disk $B_{\rho}(f(0) , \beta L(f(T_2)) )$. In conclusion,
\[ \ell_{\rho}(0,r) \geq \beta L(f(T_2)) \geq \beta ^2 L(f(T_1)) \geq \beta ^2 L_{\rho}(0,r).\]
Since this is true for all $r\leq 1$, we see that $f$ is quasiconformal at $0$ with linear distortion bounded above by $1/\beta^2$.

If $f$ does not fix $0$, then consider any $z\in \D$ with image $f(z)$. Find M\"obius maps $A_1,A_2$ which map $z$ and $f(z)$ to $0$ respectively and apply the above argument to $A_2 \circ f \circ A_1^{-1}$ to see that $f$ is quasiconformal at $z$. Since $z$ was arbitrary and the bound on the linear distortion is independent of $z$, the proof is complete.
\end{proof}

\section{Quasiconformal implies Quasisymmetric}

In this section, we will prove Theorem \ref{thm:2}. The main idea in proving quasiconformal implies quasisymmetric in the hyperbolic ball is to split  the proof into two cases. On large scales, quasiconformal maps are bi-Lipschitz by Theorem \ref{thm:qc1}, whereas on small scales quasiconformal maps are quasisymmetric by Theorem \ref{thm:qc2}. We just need to be a little careful in combining these two results.

Throughout this section, we fix $n\geq 2$ and equip the unit ball $\B^n$ in $\R^n$ with the hyperbolic metric $\rho$.

\begin{lemma}
\label{lem:qcqs1}
Suppose that $f:\B^n \to \B^n$ is $K$-quasiconformal, $f$ fixes $0$ and $t>0$. Then there exists a constant $\eta$ depending only on $t,n$ and $K$ so that
\[ \frac{ L_{\rho}(0,tr) }{ \ell_{\rho}(0,r) }  \leq \eta \]
for all $r>0$. 
\end{lemma}

\begin{proof}
Throughout the proof, we will denote $L_{\rho}(0,tr)$ and $\ell_{\rho}(0,r)$ by $L_{\rho}(tr)$ and $\ell_{\rho}(r)$ respectively.  We will denote by $x$ a point with $\rho(0,x) = tr$ and $\rho(0,f(x)) = L_{\rho}(tr)$ and by $y$ a point with $\rho(0,y) = r$ and $\rho(0,f(y)) = \ell_{\rho}(r)$.

Observe that if $f:\B^n \to \B^n$ is $K$-quasiconformal and fixes $0$, then the image of the ball centred at $0$ of hyperbolic radius $1$ is contained in the ball centred at $0$ of hyperbolic radius $C_1$ by Theorem \ref{thm:qc1}. We may assume that $C_1 \geq 1$. Then if $x,y\in B_{\rho}(0,1)$ it follows that $f(x),f(y) \in B_{\rho}(0,C_1)$. Since the Euclidean and hyperbolic metrics are bi-Lipschitz equivalent on compact subsets of $\B^n$, there exists a constant $C_3$ depending only on $n$ and $K$  so that the Euclidean and hyperbolic metrics are $C_3$-bi-Lipschitz equivalent on $B_{\rho}(0,C_1)$.
Moreover, we can apply Theorem \ref{thm:qc2} to $B(0,\widetilde{C_1}):= B_{\rho}(0,C_1)$ contained in $\B^n$, that is with $\lambda = \widetilde{C_1}$. Thus we may conclude $f$ is $\xi$-quasisymmetric on $B(0,\widetilde{C_1})$, where $\xi$ depends only on $n$ and $K$, since $\widetilde{C_1}$ depends only on $C_1$ which depends only on $n$ and $K$.

Putting all this together, if $x,y\in B_{\rho}(0,1)$, we have
\[ \frac{L_{\rho}(tr)}{\ell_{\rho}(r)} = \frac{\rho(0,f(x)) }{\rho(0,f(y))} \leq C_3^2 \frac{|f(x)|}{|f(y)|}
\leq C_3^2 \xi \left ( \frac{|x|}{|y|} \right ) \leq C_3^2\xi ( C_3^2 t) . \]

We now deal with the cases where at least one of $x,y$ are not in $B_{\rho}(0,1)$. First, suppose $t\geq 1$, so $|x| \geq |y|$, and $\rho(0,x) = tr \geq 1$. Then $r\geq 1/t$ and so
\[ r^K = r^{K-1}r \geq \frac{r}{t^{K-1}}.\]
Consequently,
\[ \min \{ r^K,r\} \geq \min \left \{ \frac{r}{t^{K-1}} ,  r \right \} = \frac{r}{t^{K-1}}.\]
By Theorem \ref{thm:qc1} and since $tr \geq 1$, it follows that
\[  \frac{L_{\rho}(tr)}{\ell_{\rho}(r)} = \frac{\rho(0,f(x)) }{\rho(0,f(y))} 
\leq \frac{C_1 tr}{C_2rt^{1-K}} = \frac{C_1t^K}{C_2}.\]

Second, suppose $t\leq 1$, so $|y| \geq |x|$, and $\rho(0,y) = r \geq 1$ since we have assumed at least one of $x,y$ are not in $B_{\rho}(0,1)$. Then $tr\geq t$ and so
\[ (tr)^{1/K} = (tr)(tr)^{1/K - 1} \leq (tr)t^{1/K-1} = t^{1/K}r.\]
Consequently,
\[ \max \{tr , (tr)^{1/K} \} \leq t^{1/K}r.\]
By Theorem \ref{thm:qc1} and since $r\geq 1$, it follows that
\[ \frac{L_{\rho}(tr)}{\ell_{\rho}(r)} = \frac{\rho(0,f(x)) }{\rho(0,f(y))} \leq \frac{C_1t^{1/K}r}{C_2r} = \frac{C_1t^{1/K}}{C_2}.\]

Combining the above estimates, we see that for any $r>0$,
\[ \frac{ L_{\rho}(rt) }{\ell_{\rho}(r) } \leq \eta:= \max \left \{ C_3^2 \xi( C_3^2t) , \frac{C_1t^K}{C_2} , \frac{C_1t^{1/K}}{C_2}  \right \},\]
and recall that $\xi,C_1,C_2,C_3$ depend only on $n$ and $K$. 
\end{proof}

We may now prove our main result of the section.

\begin{proof}[Proof of Theorem \ref{thm:2}]
Suppose that $x,y,z \in \B^n$ with $\rho(x,y) = t\rho(x,z)$ for some $t>0$. Choose M\"obius mappings $P,Q$ from $\B^n$ onto itself which map $x$ to $0$ and $f(x)$ to $0$ respectively. Denote by $\widetilde{f}$ the map $Q \circ f \circ P^{-1}$. Since M\"obius mappings are hyperbolic isometries, we have by applying Lemma \ref{lem:qcqs1} to $\widetilde{f}$ that there exists a homeomorphism $\eta:(0,\infty)\to(0,\infty)$ such that
\begin{align*}
\frac{ \rho(f(x),f(y)) }{\rho(f(x),f(z)) } &= \frac{ \rho( 0,Q(f(y)))}{ \rho(0,Q(f(z)))} \\
&= \frac{ \rho(0, \widetilde{f}(P(y)) ) }{ \rho ( 0, \widetilde{f}(P(z)))}\\
&\leq \eta \left ( \frac{ \rho ( 0 , P(y) ) }{ \rho(0,P(z) ) } \right ) \\
&= \eta \left ( \frac{ \rho ( x,y) } {\rho( x,z) } \right ).
\end{align*}
This shows that $f$ is quasisymmetric with respect to the hyperbolic metric, with quasisymmetry provided by the homeomorphism $\eta$.
\end{proof}

In the proof of Lemma \ref{lem:qcqs1} we used Theorem \ref{thm:qc2} on scales with hyperbolic distance at most $1$. If instead we had used \cite[Theorem 1.1]{FN} on small  enough scales for it to be applicable, and modified the proof so the cases where $x $ or $y $ are not in $B_\rho (0, 1) $  become the cases where Theorem 1.1  does not apply, we could directly see that we can take $\eta$ to be $\eta(t) = C \max \{ t^K , t^{1/K} \}$, where $C$ is a constant depending only on $n$ and $K$. The proof of Theorem \ref{thm:2} then implies that a quasiconformal map $f:\B^n \to \B^n$ is power quasisymmetric. We finally prove consequences of Theorem \ref{thm:2}.

\begin{proof}[Proof of Corollary \ref{cor:1}]
Let $n\geq 2$ and let $M^n,N^n$ be hyperbolic $n$-manifolds with hyperbolic distance functions $\rho_M,\rho_N$ respectively. If $f:M^n\to N^n$ is $\eta$-quasisymmetric, then it follows from the Metric Definition of quasiconformality, see Definition \ref{def:metric}, that $f$ is quasiconformal since quasiconformality is a local condition. 

On the other hand, suppose that $f:M^n\to N^n$ is $K$-quasiconformal. Writing $\pi_M,\pi_N$ for covering maps from the universal cover $\B^n$ onto $M^n,N^n$ respectively, we can lift $f$ to a $K$-quasiconformal map $\widetilde{f}:\B^n \to \B^n$ satisfying $f\circ \pi_M = \pi_N \circ \widetilde{f}$.

Let $p,q,r$ be three points in $M^n$ and choose $u,v,w\in \B^n$ with $\pi_M(u) = p, \pi_M(v) = q, \pi_M(w)=r$ and, moreover, $\rho_M(p,q) = \rho(u,v)$ and $\rho_M(p,r) = \rho(u,w)$. By Theorem \ref{thm:2}, there exists $\widetilde{\eta}$ depending only on $K$ and $n$ so that
\begin{equation}
\label{eq:cor4eq1}
\frac{ \rho( \widetilde{f}(u) , \widetilde{f}(v) )}{\rho( \widetilde{f}(u) , \widetilde{f}(w) )} \leq \widetilde{\eta} \left ( \frac{ \rho(u,v) }{\rho(u,w)} \right ) = \widetilde{\eta} \left ( \frac{ \rho_M(p,q)}{\rho_M(p,r)} \right ).
\end{equation}
Now, $\pi_N(\widetilde{f}(u)) = f(p), \pi_N(\widetilde{f}(v)) = f(q)$ and $\pi_N(\widetilde{f}(w)) = f(r)$ but we cannot assume that, for example, $\rho_N(f(p) , f(q))$ is realized by $\rho(\widetilde{f}(u) , \widetilde{f}(v) )$. However, we do have
\begin{equation}
\label{eq:cor4eq2} 
\rho_N(f(p) , f(q)) \leq \rho(\widetilde{f}(u) , \widetilde{f}(v) ).
\end{equation}
If $G_M$ is the covering group for the covering map $\pi_M :\B^n \to M^n$, then consider the orbit of $w$ under $G_M$, that is, let $\Lambda = \{g(w) : g\in G_M\}$. Then for any $w' \in \Lambda \setminus  \{ w \}$, we have
\begin{equation*}
\frac{ \rho(\widetilde{f}(u) , \widetilde{f}(w) ) }{\rho(\widetilde{f}(u) , \widetilde{f}(w') ) } \leq \widetilde{\eta} \left ( \frac{ \rho(u,w) }{\rho(u,w') } \right ) \leq \widetilde{\eta}(1),
\end{equation*}
since $\rho(u,w) = \rho_M(p,r)$ and $\widetilde{\eta}$ is increasing. Since $f\circ \pi_M = \pi_N \circ \widetilde{f}$, it follows that $\rho_N(f(p),f(r))$ is realized by the infimum of $\rho(\widetilde{f}(u) , \widetilde{f}(w') )$ as $w'$ ranges over $\Lambda$.
We therefore have
\begin{equation}
\label{eq:cor4eq3} 
\rho_N(f(p) , f(r) ) \geq \frac{ \rho ( \widetilde{f}(u) , \widetilde{f}(w) ) }{\widetilde{\eta}(1) }.
\end{equation}
By combining \eqref{eq:cor4eq1}, \eqref{eq:cor4eq2} and \eqref{eq:cor4eq3}, we conclude that
\[ \frac{ \rho_N(f(p) , f(q)) }{\rho_N(f(p),f(r) )} \leq \frac{ \widetilde{\eta}(1) \rho (\widetilde{f}(u) , \widetilde{f}(v) )}{\rho ( \widetilde{f}(u) , \widetilde{f}(w) )} \leq \widetilde{\eta}(1) \widetilde{\eta} \left ( \frac{ \rho_M(p,q)}{\rho_M(p,r)} \right ).\]
The result now follows by taking the quasisymmetry function $\eta$ to be $\widetilde{\eta}(1)\widetilde{\eta}(t)$.
\end{proof}

\begin{proof}[Proof of Corollary \ref{cor:2}]
If $f:M^n\to N^n$ is quasiconformal, then by Corollary \ref{cor:1} $f$ is $\eta$-quasisymmetric. It follows that $f$ satisfies the skew condition with constant $\eta(1)$.

Conversely, if $f:M^n \to N^n$ satisfies the skew condition with constant $\sigma$, then while we cannot necessarily guarantee the lift $\widetilde{f}$ of $f$ to $\B^n$ does, it does on small enough scales which will be enough to conclude quasiconformality. 

More precisely, if $p\in M^n$, find $u\in \B^n$ and $\delta>0$ so that the covering map $\pi_M$ is an isometry from $B_{\B^n}(u,\delta)$ onto $B_{M^n}(p,\delta)$. Then every equilateral triangle in $B_{M^n}(p,\delta)$ lifts to an equilateral triangle in $B_{\B^n}(u,\delta)$. The proof of Theorem \ref{thm:1} then implies that $\widetilde{f}$ is quasiconformal in $B_{\B^n}(u,\delta)$ with distortion bounded above by a constant depending only on $\sigma$. Hence $f$ is quasiconformal in a neighborhood of $p$ with the same distortion bound. Repeating this argument over all points in $M^n$ proves the claim.
\end{proof}

\begin{proof}[Proof of Corollary \ref{cor:3}]
This is immediate from Corollary \ref{cor:1}, since $\eta$ only depends on $K$ and $n$. 
\end{proof}

\end{document}